\documentclass[reqno,11pt]{amsart}

\usepackage{euscript}
\usepackage{enumerate}
\usepackage{amsmath}
\usepackage[initials]{amsrefs}
\usepackage{amssymb}
\usepackage{comment}

\theoremstyle{plain}
\newtheorem{thm}{Theorem}[section]
\newtheorem{lemma}[thm]{Lemma}
\newtheorem{prop}[thm]{Proposition}
\newtheorem{ques}[thm]{Question}

\newtheorem{cor}[thm]{Corollary}

\theoremstyle{remark}
\newtheorem{remark}[thm]{Remark}

\usepackage{color}

\newcount\theTime
\newcount\theHour
\newcount\theMinute
\newcount\theMinuteTens
\newcount\theScratch
\theTime=\number\time
\theHour=\theTime
\divide\theHour by 60
\theScratch=\theHour
\multiply\theScratch by 60
\theMinute=\theTime
\advance\theMinute by -\theScratch
\theMinuteTens=\theMinute
\divide\theMinuteTens by 10
\theScratch=\theMinuteTens
\multiply\theScratch by 10
\advance\theMinute by -\theScratch

\def\today{{\number\day\space
 \ifcase\month\or
  January\or February\or March\or April\or May\or June\or
  July\or August\or September\or October\or November\or December\fi
 \space\number\year}}

\newcommand\Cpx{{\mathbf C}}
\newcommand\HEu{{\EuScript H}}                   
\newcommand\Ic{{\mathcal{I}}}
\newcommand\KEu{{\EuScript K}}                   
\newcommand\ran{\operatorname{ran}}
\newcommand\Reals{{\mathbf R}}
\newcommand\Vc{{\mathcal{V}}}
\newcommand\Wc{{\mathcal{W}}}

\begin{document}

\title[Nilpotents as commutators]{Nilpotent elements of operator ideals as single commutators}

\author[Dykema]{Ken Dykema}
\address{Ken Dykema, Department of Mathematics, Texas A\&M University, College Station, TX, USA.}
\email{ken.dykema@math.tamu.edu}
\author[Krishnaswamy--Usha]{Amudhan Krishnaswamy--Usha}
\address{Amudhan Krishnaswamy--Usha, Department of Mathematics, Texas A\&M University, College Station, TX, USA.}
\email{amudhan@math.tamu.edu}
\begin{abstract}
For an arbitrary operator ideal $\Ic$, every nilpotent element of $\Ic$ is a single commutator
of operators from $\Ic^{\,t}$, for an exponent $t$ that depends on the degree of nilpotency.
\end{abstract}

\date{June 11, 2017}

\maketitle

\section{Introduction}

By {\em operator ideal} we mean a proper, nonzero, two-sided ideal
of the algebra $B(\HEu)$ of bounded operators on a separable, infinite Hilbert space $\HEu$.
These ideals consist of compact operators.
For a compact operator, $A$ on $\HEu$, let $s(A)=(s_1(A),s_2(A),\ldots)$ be the
sequence of singular numbers of $A$.
This is the non-increasing sequence of nonzero eigenvalues of $|A|:=(A^*A)^{1/2}$, listed in order of multiplicity, 
with a tail of zeros in case $A$ has finite rank.
As Calkin showed~\cite{C41}, an operator ideal $\Ic$ is characterized by $s(\Ic)=\{s(A)\mid A\in\Ic\}$.
(See also, e.g.,~\cite{GK69} or~\cite{DFWW04} for expositions).
For a positive real number $t$ and an operator ideal $\Ic$, we let $\Ic^{\,t}$ denote the operator ideal generated
by $\{|A|^t\mid A\in\Ic\}$.

Questions about additive commutators $[B,C]:=BC-CB$ involving elements of operator ideals have been much studied.
One of the questions asked in~\cite{PT71}, by Pearcy and Topping, is whether every compact operator $A$ is a single commutator
$A=[B,C]$ of compact operators $B$ and $C$.
This question is still open.
Important results about single commutators in operator ideals were obtained in~\cite{PT71} and by Anderson~\cite{A77}.
Further results are found in Section~7 of~\cite{DFWW04}.
More recently, Belti\c t\u a, Patnaik and Weiss~\cite{BPW14} have made progress on the above mentioned question.

Our purpose in this note is to show that every nilpotent compact operator is a single commutator of compact operators.
In fact, we show (Theorem~\ref{thm:n}) that for a general operator ideal $\Ic$, every nilpotent element $A\in\Ic$ is a single commutator
$A=[B,C]$ of $B,C\in\Ic^{\,t}$, where the value of $t>0$ depends on the value of $n$ for which $A^n=0$.
Except in the case $n\le 4$, we don't know if we have found the optimal value of $t$.

\section{Preliminaries}

Let $\HEu$ be an infinite dimensional Hilbert space.
Nothing in this section is new, but we include proofs for convenience.

\begin{lemma}\label{lem:yx}
Suppose $x,y\in B(\HEu)$ and $t\in\Reals$, $t>0$.
\begin{enumerate}[(i)]
\item\label{it:yrx}
If $t(x^*x)\ge y^*y$, then there exists $r\in B(\HEu)$ such that $\|r\|\le\sqrt t$ and $y=rx$.
\item\label{it:yxr}
If $t(xx^*)\ge yy^*$, then there exists $r\in B(\HEu)$ such that $\|r\|\le\sqrt t$ and $y=xr$.
\end{enumerate}
\end{lemma}
\begin{proof}
The assertion~\eqref{it:yxr} follows from~\eqref{it:yrx} by taking adjoints.
If we prove the assertion~\eqref{it:yrx} when $t=1$, then the case of arbitrary $t$ follows, by replacing $x$ with $\sqrt tx$.
So it will prove~\eqref{it:yrx} in the case $t=1$.

Suppose $x^*x\ge y^*y$.
Given $\xi\in\HEu$, we have
\[
\|y\xi\|^2=\langle y^*y\xi,\xi\rangle\le\langle x^*x\xi,\xi\rangle=\|x\xi\|^2.
\]
Thus, we may define a contractive linear operator from $\ran(x)$ into $\HEu$ by
\[
x\xi\mapsto y\xi.
\]
This extends uniquely to a contractive linear operator, which we call $r_0$, from $\overline{\ran(x)}$ into $\HEu$.
We have $r_0x=y$.
Letting $p$ be the orthogonal projection from $\HEu$ onto $\overline{\ran(x)}$, we set $r=r_0p$.
Thus, $r\in B(\HEu)$ is a contraction and $rx=y$.
\end{proof}

For $n\ge1$, we make the natural identifications
\begin{equation}\label{eq:BHn}
B(\HEu^{\oplus n})=M_n(B(\HEu))=B(\HEu)\otimes M_n(\Cpx)
\end{equation}
and we let $(e_{i,j})_{1\le i,j,\le n}$ be the usual system of matrix units in $M_n(\Cpx)$.

Recall that $\HEu$ is assumed to be infinite dimensional
(and here we do not need to assume it is separable.) 
\begin{lemma}\label{lem:nilpUT}
Let $A\in B(\HEu)$ satisfy $A^n=0$.
Then there exists a unitary $U:\HEu\to\HEu^{\oplus n}$ such that $UAU^*$ is a strictly upper triangular element of $M_n(B(\HEu))$.
\end{lemma}
\begin{proof}
We will first show that $\dim\ker A=\dim\HEu$.
Consider $B= A|_{\ker A^2}$.
Note that $\ker A=\ker B$.
If $\dim\ker A^2=\dim\HEu$, then either $\dim\ker B=\dim\HEu$, or $\dim\ran B=\dim\HEu$.
But $\ran B\subset \ker B$, so $\dim\ker A^2=\dim\HEu$ implies $\dim\ker A=\dim\HEu$.
Since $A$ is nilpotent, we have $\dim\ker A^{2^k}=\dim\HEu$, for some $k$.
Thus, (arguing by induction on $k$), we must have $\dim\ker A=\dim\HEu$.

Let
\begin{align*}
\Vc_1&=\ker A, \\
\Vc_k&=\ker A^k\ominus\ker A^{k-1},\quad(2\le k\le n) \\
\end{align*}
We will construct subspaces
\[
\Wc_1\subseteq\Wc_2\subseteq\cdots\subseteq\Wc_n=\HEu
\]
with
\[
\Wc_k\subseteq\ker A^k
\]
such that, letting $\Wc_0=\{0\}$, we have, for every $1\le k\le n$,
\begin{equation}\label{eq:Wk}
\dim(\Wc_k\ominus\Wc_{k-1})=\dim\HEu
\end{equation}
and for every $k\le n-1$,
\begin{align}
\dim((\ker A^{k+1})\ominus\Wc_k)&=\dim\HEu, \label{eq:Wcond} \\
A(\Vc_{k+1})&\subseteq\Wc_k. \label{eq:AV}
\end{align}

Fixing $k=1$,
if $\dim\Vc_2=\dim\HEu$, then let $\Wc_1=\ker A$.
We know $\dim\ker A=\dim\HEu$, so~\eqref{eq:Wk} holds.
Moreover, $\ker A^2\ominus\Wc_1=\Vc_2$, so~\eqref{eq:Wcond} holds
and $A(\Vc_2)\subseteq A(\ker A^2)\subseteq\ker A$, so~\eqref{eq:AV} holds.
Otherwise, if $\dim\Vc_2<\dim\HEu$, then choose $\Wc_1$ so that 
\[
A(\Vc_2)\subseteq\Wc_1\subseteq\ker A
\]
and
\[
\dim\Wc_1=\dim\HEu=\dim(\ker A\ominus\Wc_1).
\]
This choice is possible because we know $\dim\ker A=\dim\HEu$ and by hypothesis $\dim A(\Vc_2)\le\dim\Vc_2<\dim\HEu$.
Then~\eqref{eq:Wk} and~\eqref{eq:AV} (for $k=1$) hold by
construction.
We have
\[
\dim\HEu\ge\dim((\ker A^2)\ominus\Wc_1)\ge\dim((\ker A)\ominus\Wc_1)=\dim\HEu,
\]
so~\eqref{eq:Wcond} holds.

Now suppose $2\le k\le n-1$ and $\Wc_1,\ldots,\Wc_{k-1}$ have been constructed with the required properties.
If $\dim\Vc_{k+1}=\dim\HEu$, then let $\Wc_k=\ker A^k$.
Then~\eqref{eq:Wk} for $k$ is just~\eqref{eq:Wcond} for $k-1$ while~\eqref{eq:Wcond} for $k$ is just the hypothesis $\dim(\Vc_{k+1})=\dim\HEu$.
Moreover, $A(\Vc_{k+1})\subseteq A(\ker A^{k+1})\subseteq\ker A^k$, so~\eqref{eq:AV} holds for this $k$ as well.

Othewrwise, if $\dim\Vc_{k+1}<\dim\HEu$, then choose $\Wc_k$ so that
\[
A(\Vc_{k+1})+\Wc_{k-1}\subseteq\Wc_k\subseteq\ker A^k
\]
and
\[
\dim(\Wc_k\ominus\Wc_{k-1})=\dim\HEu=\dim((\ker A^k)\ominus\Wc_k).
\]
This is possible because, by hypothesis (namely, \eqref{eq:Wcond} for $k-1$),
\[
\dim((\ker A^k\ominus\Wc_{k-1})=\dim\HEu
\]
and $\dim(A(\Vc_{k+1}))\le\dim\Vc_{k+1}<\dim\HEu$.
Then~\eqref{eq:Wk} and~\eqref{eq:AV} hold by construction, while for~\eqref{eq:Wcond}, we use
\[
\dim\HEu\ge\dim((\ker A^{k+1})\ominus\Wc_k)\ge\dim((\ker A^k)\ominus\Wc_k)=\dim\HEu.
\]

Finally, set $\Wc_n=\HEu=\ker A^n$.
Then~\eqref{eq:Wk} for $k=n$ follows from~\eqref{eq:Wcond} for $k=n-1$.

Using~\eqref{eq:AV}, we get
\[ 
A(\Wc_k) \subseteq A(\ker A^k) = A(\Vc_1)+\cdots+A(\Vc_k) \subseteq \Wc_{k-1}.
\]
Let $\HEu_1 = \Wc_1$, and $\HEu_k = \Wc_{k}\ominus \Wc_{k-1}$, $2\le k\leq n$.
Then $\dim\HEu_k=\dim\HEu$ for all $k$ and
\begin{align*}
A(\HEu_1)&=\{0\} \\
A(\HEu_k)&\subseteq\bigoplus_{j=1}^{k-1}\HEu_j,\quad(2\le k\le n).
\end{align*}
Choosing unitaries $U_k:\HEu_k\to\HEu$
yields a unitary $U=\oplus_{k=1}^nU_j:\HEu\to\HEu^{\oplus n}$ so that $UAU^*$ is a strictly upper triangular matrix. 
\end{proof}

\begin{remark}\label{rem:ABC}
We work in $B(\HEu)\otimes M_n(\Cpx)$ and suppose
\[
A=\sum_{1\le i<j\le n}a_{i,j}\otimes e_{i,j}
\]
for $a_{i,j}\in B(\HEu)$.
If
\[
B=\sum_{i=1}^{n-1}b_i\otimes e_{i,i+1},\qquad C=\sum_{2\le i\le j\le n}c_{i,j}\otimes e_{i,j}
\]
with $b_i,c_{i,j}\in B(\HEu)$,
then the condition $A=BC-CB$,
is equivalent to 
\begin{alignat*}{2}
a_{1,j}&=b_1\,c_{2,j}\quad&&(2\le j\le n) \\
a_{i,j}&=b_i\,c_{i+1,j}-c_{i,j-1}b_{j-1}\quad&&(2\le i<j\le n)
\end{alignat*}
or, equivalently,
\begin{alignat}{2}
b_1\,c_{2,j}&=a_{1,j}\quad&&(2\le j\le n) \label{eq:bc2} \\
b_i\,c_{i+1,j}&=a_{i,j}+c_{i,j-1}b_{j-1}\quad&&(2\le i<j\le n). \label{eq:bc} 
\end{alignat}
\end{remark}

\section{Nilpotents in operator ideals}

Let $\Ic$ be an operator ideal.
It is well known and easy to see that, under any identification of $B(\HEu)$ with $M_n(B(\HEu))$ as in~\eqref{eq:BHn}, the ideal $\Ic$ is identified with $M_n(\Ic)$.

We first prove the following easy result, whose proof is similar to that of Proposition~3.2 of~\cite{DS12}.
\begin{prop}\label{prop:nilp}
Let $\Ic$ be an operator ideal and suppose $A\in\Ic$ is nilpotent.
Then there exist $B\in B(\HEu)$ and $C\in\Ic$ such that $A=BC-CB$.
\end{prop}
\begin{proof}
Let $n\ge2$ be such that $A^n=0$.
By Lemma~\ref{lem:nilpUT}, we may work in $B(\HEu)\otimes M_n(\Cpx)$ and suppose
\[
A=\sum_{1\le i<j\le n}a_{i,j}\otimes e_{i,j}
\]
for $a_{i,j}\in\Ic$.
We need only find elements $b_i\in B(\HEu)$ and $c_{i,j}\in\Ic$, as in Remark~\ref{rem:ABC}, so that~\eqref{eq:bc2} and~\eqref{eq:bc} hold.
This is easily done by setting $b_i=1$ for all $i$ and recursively assigning
\begin{alignat*}{2}
c_{2,j}&=a_{1,j}\quad&&(2\le j\le n)\\
c_{i+1,j}&=a_{i,j}+c_{i,j-1},\quad&&(2\le i<j\le n).
\end{alignat*}
\end{proof}

\begin{thm}\label{thm:n}
Let $\Ic$ be an operator ideal and suppose
$A\in\Ic$ satisfies $A^n=0$, for some integer $n\ge4$.
Then there exist $B,C\in\Ic^{1/2^{n-3}}$ such that $A=BC-CB$.
\end{thm}
\begin{proof}
By Lemma~\ref{lem:nilpUT}, we may work in $B(\HEu)\otimes M_n(\Cpx)$ and suppose
\[
A=\sum_{1\le i<j\le n}a_{i,j}\otimes e_{i,j}
\]
for $a_{i,j}\in\Ic$.
We will find elements $b_i$ and $c_{i,j}$ of $\Ic^{1/2^{n-3}}$, as in Remark~\ref{rem:ABC}, so that~\eqref{eq:bc2} and~\eqref{eq:bc} hold.

\vskip1ex
\noindent
{\em Step 1: assign values to $b_1,\ldots,b_{n-2}$.}
\vskip1ex
Let
\begin{align*}
b_1&=\left(\sum_{j=2}^n|a_{1,j}^*|^2\right)^{1/4}\in\Ic^{1/2} \\
b_i&=\left(b_{i-1}^2+\sum_{j=i+1}^n|a_{i,j}^*|^2\right)^{1/4}\in\Ic^{1/2^i},\qquad(2\le i\le n-3) \\
b_{n-2}&=\left(b_{n-3}^4+\sum_{j=i+1}^n|a_{i,j}^*|^2\right)^{1/4}\in\Ic^{1/2^{n-3}}.
\end{align*}
Since for every $1\le i\le n-2$ and every $i<j\le n$, we have $b_i^4\ge|a_{i,j}^*|^2$, by Lemma~\ref{lem:yx} there exists $r_{i,j}\in B(\HEu)$
such that
\[
b_i^2r_{i,j}=a_{i,j}\qquad(1\le i\le n-2,\,i<j\le n).
\]
Moreover, for every $2\le i\le n-3$, since $b_i^4\ge b_{i-1}^2$,
by the same lemma there exists $x_i\in B(\HEu)$ such that
\[
b_i^2x_i=b_{i-1}\qquad(2\le i\le n-3).
\]
Furthermore, since $b_{n-2}^4\ge b_{n-3}^4$ and the square root function is operator monotone, we have $b_{n-2}^2\ge b_{n-3}^2$.
Thus, by Lemma~\ref{lem:yx} there exists $z\in B(\HEu)$ so that 
\[
b_{n-2}\,z=b_{n-3}.
\]

\vskip1ex
\noindent
{\em Step 2: assign values to $y_{2,j}$ and $c_{2,j}$ for $2\le j\le n$ and verify~\eqref{eq:bc2}.} 
\vskip1ex
Let
\[
y_{2,j}=r_{1,j},\qquad c_{2,j}=b_1y_{2,j},\qquad(2\le j\le n).
\]
Thus, $c_{2,j}\in\Ic^{1/2}$.
Then we have
\[
b_1c_{2,j}=b_1^2r_{1,j}=a_{1,j},\qquad(2\le j\le n),
\]
namely,~\eqref{eq:bc2} holds.

\vskip1ex
\noindent
{\em Step 3: assign values to $y_{p,j}$ and $c_{p,j}$ for $3\le p\le n-2$ and $p\le j\le n-1$ and verify the equality in~\eqref{eq:bc} for
$2\le i\le n-3$ and $i<j\le n-1$.} 
\vskip1ex
We let $p$ increase from $3$ to $n-2$ and for each such $p$ we define (recursively in $p$) for every $j\in\{p,p+1,\ldots,n-1\}$,
\[
y_{p,j}=r_{p-1,j}+x_{p-1}y_{p-1,j-1}b_{j-1},\qquad c_{p,j}=b_{p-1}y_{p,j}.
\]
Thus, $c_{p,j}\in\Ic^{1/2^{p-1}}$ and
we have 
\[
\begin{aligned}[b]
b_ic_{i+1,j}&=b_i^2r_{i,j}+b_i^2x_iy_{i,j-1}b_{j-1} \\
&=a_{i,j}+b_{i-1}y_{i,j-1}b_{j-1} \\
&=a_{i,j}+c_{i,j-1}b_{j-1},
\end{aligned}
\qquad
(2\le i\le n-3,\,i<j\le n-1)
\]
and the equality in~\eqref{eq:bc} holds for these values of $i$ and $j$.

\vskip1ex
\noindent
{\em Step 4: assign a value to $c_{n-1,n-1}$ and verify the equality in~\eqref{eq:bc} for $i=n-2$ and $j=n-1$.}
\vskip1ex
Let
\[
c_{n-1,n-1}=b_{n-2}r_{n-2,n-1}+zy_{n-2,n-2}b_{n-2}.
\]
Then $c_{n-1,n-1}\in\Ic^{1/2^{n-3}}$ and
\begin{align*}
b_{n-2}c_{n-1,n-1}&=b_{n-2}^2r_{n-2,n-1}+b_{n-2}zy_{n-2,n-2}b_{n-2} \\
&=a_{n-2,n-1}+b_{n-3}y_{n-2,n-2}b_{n-2} \\
&=a_{n-2,n-1}+c_{n-2,n-2}b_{n-2}.
\end{align*}
Thus, the equality in~\eqref{eq:bc} holds for $i=n-2$ and $j=n-1$.

\vskip1ex
\noindent
{\em Step 5: assign a value to $b_{n-1}$.}
\vskip1ex
Let
\[
b_{n-1}=\left(|a_{n-1,n}^*|^2+|c_{n-1,n-1}^*|^4\right)^{1/4}.
\]
Then $b_{n-1}\in\Ic^{1/2^{n-3}}$.
Since $b_{n-1}^4\ge|a_{n-1,n}^*|^2$, by Lemma~\ref{lem:yx} there is $r_{n-1,n}\in B(\HEu)$ so that
\[
b_{n-1}^2r_{n-1,n}=a_{n-1,n}.
\]
Since $b_{n-1}^4\ge|c_{n-1,n-1}^*|^4$, we have $b_{n-1}^2\ge|c_{n-1,n-1}^*|^2$ and, from Lemma~\ref{lem:yx}, we have $s\in B(\HEu)$ so that
\[
b_{n-1}s=c_{n-1,n-1}.
\]

\vskip1ex
\noindent
{\em Step 6: assign values to $c_{p,n}$ for all $3\le p\le n-2$ and verify the equality in~\eqref{eq:bc} for all $2\le i\le n-3$ and $j=n$.}
\vskip1ex
Let
\[
c_{p,n}=b_{p-1}r_{p-1,n}+b_{p-1}x_{p-1}y_{p-1,n-1}b_{n-1}.
\]
Then $c_{p,n}\in\Ic^{1/2^{p-1}}$ and 
\[
\begin{aligned}[b]
b_ic_{i+1,n}&=b_i^2r_{i,n}+b_i^2x_iy_{i,n-1}b_{n-1} \\
&=a_{i,n}+b_{i-1}y_{i,n-1}b_{n-1} \\
&=a_{i,n}+c_{i,n-1}b_{n-1},
\end{aligned}
\qquad(2\le i\le n-3),
\]
namely, the equality in~\eqref{eq:bc} holds for these values of $i$ and for $j=n$.

\vskip1ex
\noindent
{\em Step 7: assign a value to $c_{n-1,n}$ and verify the equality in~\eqref{eq:bc} for $i=n-2$ and $j=n$.}
\vskip1ex
Let
\[
c_{n-1,n}=b_{n-2}r_{n-2,n}+zy_{n-2,n-1}b_{n-1}.
\]
Then $c_{n-1,n}\in\Ic^{1/2^{n-3}}$ and
\begin{align*}
b_{n-2}c_{n-1,n}&=b_{n-2}^2r_{n-2,n}+b_{n-2}zy_{n-2,n-1}b_{n-1} \\
&=a_{n-2,n}+b_{n-3}y_{n-2,n-1}b_{n-1} \\
&=a_{n-2,n}+c_{n-2,n-1}b_{n-1},
\end{align*}
namely, the equality in~\eqref{eq:bc} holds for $i=n-2$ and for $j=n$.

\vskip1ex
\noindent
{\em Step 8: assign a value to $c_{n,n}$ and verify the equality in~\eqref{eq:bc} for $i=n-1$ and $j=n$.}
\vskip1ex
Let
\[
c_{n,n}=b_{n-1}r_{n-1,n}+sb_{n-1}.
\]
Then $c_{n,n}\in\Ic^{1/2^{n-3}}$ and
\[
b_{n-1}c_{n,n}=b_{n-1}^2r_{n-1,n}+b_{n-1}sb_{n-1}=a_{n-1,n}+c_{n-1,n-1}b_{n-1},
\]
as required.
\end{proof}

\begin{cor}\label{cor:ABC}
Let $\Ic$ by any operator ideal such that $\Ic^{\,t}\subseteq\Ic$ for every $t>0$.
Then for every nilpotent element $A$ of $\Ic$, there exist $B,C\in\Ic$ such that $A=BC-CB$.
\end{cor}
Examples of operator ideals $\Ic$ satisfying the conditions of Corollary~\ref{cor:ABC} include
\begin{enumerate}[(a)]
\item the ideal  $\KEu$ of all compact operators;
\item the ideal of all operators $A$ whose singular numbers have polynomial decay: $s_n(A)=O(n^{-t})$ for some $t>0$;
note that this ideal is equal to the union of all Schatten $p$-class ideals, $p\ge1$;
\item the ideal of all operators $A$ whose singular numbers have exponential decay: $s_n(A)=O(r^n)$ for some $0<r<1$;
\item the ideal of all finite rank operators.
\end{enumerate}

\begin{ques}
Is $1/2^{n-3}$ the optimal exponent of $\Ic$ in Theorem~\ref{thm:n}?
Clearly, the answer is yes when $n=4$.
But as far as we know, it is possible that the best exponent is $1/2$ for arbitrary $n$.
\end{ques}

\end{document}